\documentclass[12pt]{amsart}
\usepackage{amsmath,amsfonts,amssymb,amscd,euscript}
\usepackage{enumerate,url,verbatim,calc,xypic}
\usepackage{amsmath}

\newcommand{\edit}[1]{}

 \usepackage{latexsym}
 \usepackage{graphics}
 \usepackage{color}

\def\map#1{{\;\buildrel #1 \over \longrightarrow}\;}
\newcommand{\Spec}{\operatorname{Spec}}
\newcommand{\Pic}{\operatorname{Pic}}

\newcommand{\Z}{{\mathbb Z}}
\newcommand{\R}{{\mathbb R}}
\newcommand{\C}{{\mathbb C}}
\newcommand{\cI}{{\mathcal I}}

\newcommand{\cF}{{\mathcal F}}
\newcommand{\cO}{{\mathcal O}}
\newcommand{\ca}{{\mathfrak a}}
\newcommand{\cc}{{\mathfrak c}}
\newcommand{\et}{{\mathrm{et}}}
\newcommand{\nis}{{\mathrm{nis}}}
\newcommand{\zar}{{\mathrm{zar}}}
\newcommand{\red}{{\mathrm{red}}}

\newcommand{\coker}{\mathop{\mathrm{coker}\,}}
\newcommand{\nil}{{\mathrm{nil}}}
\newcommand{\Unil}{{U_\nil}}   
\newcommand{\scA}{{}^{^+}\!\!\!A}

\numberwithin{equation}{section} 
\newtheorem{theorem}[equation]{Theorem}
\newtheorem{corollary}[equation]{Corollary}

\newtheorem{lemma}[equation]{Lemma}
\newtheorem{proposition}[equation]{Proposition}

\newtheorem{question}[equation]{Question}
\theoremstyle{definition}
\newtheorem{definition}[equation]{Definition}

\newtheorem{example}[equation]{Example}
\newtheorem{examples}[equation]{Examples}
\theoremstyle{remark}

\newtheorem{subremark}{Remark}[equation] 
\newtheorem{notation}[equation]{Notation}

\begin{document}
\title{Relative Cartier Divisors and Laurent Polynomial Extensions}
\date{\today}

\author{Vivek Sadhu}
\address{Institute of Mathematical Sciences, 
IV Cross Road, CIT Campus, Taramani, Chennai 600113, Tamil Nadu, India} 
\email{viveksadhu@imsc.res.in, viveksadhu@gmail.com}
\thanks{Sadhu was supported by IMSC, Chennai Postdoctoral Fellowship.}

\author{Charles Weibel}
\address{Math.\ Dept., Rutgers University, New Brunswick, NJ 08901, USA}
\email{weibel@math.rutgers.edu}\urladdr{http://math.rutgers.edu/~weibel}
\thanks{Weibel was supported by NSF grants} 

\begin{abstract}
If $i:A\subset B$ is a commutative ring extension, we show that the group
$\cI(A,B)$ of invertible $A$-submodules of $B$ is contracted
in the sense of Bass, with $L\cI(A,B)=H^0_\et(A,i_*\Z/\Z)$. This gives 
a canonical decomposition for $\cI(A[t,\frac1t],B[t,\frac1t])$.
\end{abstract}

\keywords{Invertible modules, Contracted functor, \'etale sheaf}
\subjclass{13B02, 13F45, 14C22}

\maketitle

\section{Introduction}
Let $A\subset B$ be a ring extension. The group $\cI(A, B)$ of
invertible $A$-submodules of $B$ is related to the Picard groups and
the units groups of $A$ and $B$ by the exact sequence
\[
 1\to  U(A) \to  U(B) \to  \cI(A,B) \to  \Pic A \to  \Pic B. 
\]
(See \cite[\S2]{rs}.) Replacing $A\subset B$ with $A[t]\subset B[t]$ and
$A[t,1/t]\subset B[t,1/t]$ yields similar exact sequences.
Following Bass \cite{bass}, each functor $F$ on rings defines functors 
$NF$ and $LF$ so that $F(A[t])=F(A)\oplus NF(A)$ and for certain
functors like $F=U$ and $\Pic$, called {\it contracted functors},
we even have a natural decomposition
\[
F(A[t,1/t])\cong F(A)\oplus NF(A)\oplus NF(A)\oplus LF(A). 
\]
%
The decompositions of $U(A[t,1/t])$ and $\Pic A[t,1/t]$ are given
in \cite[XII.7.8]{bass} and \cite{wei}.
We can define $N\cI(A,B)$ and $L\cI(A,B)$ in the same way.
Here is our main result. 

\begin{theorem}
Given a commutative ring extension $f:A\subset B$, 
$\cI$ is a contracted functor with $L\cI(A,B)=H^0_\et(\Spec A,f_*\Z/\Z)$.
In particular, there is a natural decomposition
\[
\cI(A[t,1/t],B[t,1/t]) \cong \cI(A,B)\oplus N\cI(A,B)\oplus N\cI(A,B)
\oplus L\cI(A,B),
\]
In addition, 
$
L\cI(A,B)=L\cI(A[t],B[t])=L\cI(A[t,1/t],B[t,1/t]).
$
\end{theorem}

This is proven in Theorem \ref{thm:main} and 
Proposition \ref{LU-LPic} below.
Here $\Z$ is regarded as the constant \'etale sheaf on both
$\Spec A$ and $\Spec B$, and $f_*\Z$ is the direct image sheaf on
$\Spec A$. The group $L\cI(A,B)$ also equals the Nisnevich
cohomology group $H^0_\nis(\Spec A,f_*\Z/\Z)$, but differs from
the Zariski cohomology group $H^0_\zar(\Spec A,f_*\Z/\Z)$; 
see Example \ref{node}.

For convenience, let us write $A[T]$ for $A[t_1,1/t_1,\dots,t_n,1/t_n]$.
As pointed out by Bass \cite{bass}, we can iterate the operations $N$ and $L$
to get decompositions of $\cI(A[T],B[T])$ using components
$N^iL^j\cI(A,B)$ for $1\le i,j\le n$. 
Since our Main Theorem says that $NL\cI=L^2\cI=0$,
most of these terms are unnecessary.

\begin{corollary}
For every ring extension $A\subset B$, 
$\cI(A[T],B[T])$ is the direct sum of $\cI(A,B)$,
$n$ terms of the form $L\cI(A,B)$ and 
$2^i\binom{n}{i}$ terms of the form $N^i\cI(A,B)$, $1\!\le\! i\!\le\! n$.
\end{corollary}

Since we know from \cite{ss} that $N\cI(A,B)=0$ is equivalent to
$A$ being seminormal in $B$ (Definition \ref{def:sn}) 
we can further conclude:

\begin{corollary}
For $A\subset B$, the following are equivalent:

1) $\cI(A,B)=\cI(A[t,1/t],B[t,1/t])$; 

2) $H_\et^0(\Spec(A),f_*\Z/\Z)=0$
and $A$ is seminormal in $B$.
\end{corollary}

It is immediate from our Main Theorem that $L\cI(A,B)$ is a
torsionfree group (we give a simple proof in Corollary \ref{tfree});
it is free abelian of finite rank when $A$ is pseudo-geometric 
and finite dimensional (Proposition \ref{free}).

A secondary goal of this paper is to give simple techniques
for determining $L\cI(A,B)$. For example, we may assume $A$ and $B$
are reduced as $L\cI(A,B)\cong L\cI(A_\red,B_\red)$ (Theorem \ref{LI-red}).
The following special case of Proposition \ref{connected}
gives an elementary criterion for the vanishing of $L\cI(A,B)$. 
We say that an extension $B/A$ is {\it connected} if for every prime ideal 
$\wp$ of $A$, the ring $B_{\wp}/\wp B_\wp$ is connected.

\begin{proposition}
If $B/A$ is finite and $B$ is connected over $A$ then $L\cI(A,B)=0$.
\end{proposition}

We also show that $L\cI(A,B)=0$ can only happen if the extension
$A\subset B$ is {\it anodal} in the sense of Asanuma
(Theorem \ref{ano}). The converse is true for integral, birational
extensions of 1-dimensional domains, 
but Example \ref{2 dim} (taken from \cite[3.5]{wei}) shows that being
integral, birational and anodal is not sufficient in higher dimension.

\medskip
This paper is laid out as follows.
In Section \ref{sec:prelim}, we define contracted functors on extensions
and recall some basic theory.
In Section \ref{sec:I}, we define $\cI(A,B)$ and prove 
(in Proposition \ref{LU-LPic}) that $NL\cI=LL\cI=0$.
Section \ref{sec:basic} gives some basic properties of $\cI(A,B)$.
The rest of Theorem 0.1 is proven in Section \ref{sec:main},
and Section \ref{sec:vanish} describes some conditions under which
$L\cI(A,B)$ vanishes.

\medskip
\textbf{Acknowledgements}: The first author would like to express his 
sincere thanks to Balwant Singh for many fruitful discussions 
and for his guidance, and to D.\,S.\ Nagaraj for useful discussion. 

\bigskip
\section{Contracted functors}\label{sec:prelim}
All of the rings we consider are commutative with $1$, and all ring
homomorphisms are unitary. The category of ring extensions has
objects $f:A\hookrightarrow B$; a morphism from $f$ to
$f':A'\hookrightarrow B'$ is a morphism $B\to B'$ sending $A$ to $A'$.

In \cite[XII]{bass}, Bass defined the notion of a contracted
functor from rings to abelian groups. This has a natural
translation into the setting of ring extensions, which we
now lay out.
Given an indeterminate $t$, we write $f[t]$ for the polynomial
ring extension $A[t]\hookrightarrow B[t]$ and write $f[t,1/t]$
for the Laurent polynomial extension
$A[t,1/t]\hookrightarrow B[t,1/t]$.

\begin{definition}\label{def:contracted}
Let $F$ be a functor from ring extensions to abelian groups.
We write $LF(A,B)$ or $LF(f)$ for the cokernel of the
map $F(f[t])\times F(f[1/t])\map{\pm} F(f[t,1/t])$
which is the difference of the maps induced by applying $F$
to the morphisms $f\to f[t]$ and $f\to f[1/t]$.
We write $Seq(F,f)$ for the following sequence (where
$\Delta$ is the diagonal map):
\[
1 \to F(f) \map{\Delta} F(f[t])\times F(f[1/t])\map{\pm} F(f[t,1/t])
\to LF(f)\to 1.
\]
We say $F$ is \emph{acyclic} if $Seq(F,f)$ is exact for every
ring extension $f$. We say that $F$ is \emph{contracted}
if $Seq(F,f)$ is naturally split exact, i.e., if there is a
map $h(f): LF(f)\to F(f[t,1/t])$.
\end{definition}

Following Bass \cite[XII]{bass}, we write $NF(f)$ or $N_tF(f)$
for the kernel of the map $F(f[t])\to F(f)$ induced by $t\mapsto1$.
This map is split by the map $F(f\to f[t])$ induced by 
$B\subset B[t]$, and we have a natural decomposition
$F(f[t])=F(f)\oplus NF(f)$. Thus $Seq(F,f)$ is quasi-isomorphic
to the sequence
\[
1 \to F(f)\oplus N_tF(f)\oplus N_{1/t}F(f) \to F(f[t,1/t])\to LF(f)\to1.
\]

If $F$ is a functor from rings to abelian groups, we can define
functors $s^*F(f)=F(A)$ and $t^*F(f)=F(B)$ by composing with the
source and target functors from ring extensions to rings sending
$f:A\hookrightarrow B$ to $s(f)=A$ and $t(f)=B$. If $F$ is contracted
in Bass' sense then $s^*F$ and $t^*F$ are contracted in the
sense of Definition \ref{def:contracted}.

It should be clear to the reader that the notion of contracted
functor also makes sense for functors from many categories (such as
commutative rings, schemes, ring extensions, ...) to any abelian
category (such as abelian groups, sheaves, modules).  When these
choices are irrelevant, we will not specify them and merely refer
to ``contracted functors.''

\begin{lemma}\label{mor}
Let $\eta:F\to G$ be a morphism of contracted functors.
Then $\ker(\eta)$ and $\coker(\eta)$ are contracted functors,
with $L\ker(\eta)=\ker(L\eta)$ and $L\coker(\eta)=\coker(L\eta)$.

If\ $1\to F\to G\to H\to1$ is a short exact sequence of functors
and $F$, $H$ are acyclic then $G$ is acyclic and there
is a short exact sequence
\[
1 \to LF \to LG\to LH \to1.
\]
\end{lemma}

\begin{proof}
The first assertion is proven exactly as the corresponding
assertion for contracted functors on rings in \cite[XII.7.2]{bass} 
or \cite[III.4.2]{WK}. The second assertion is proven exactly
as Carter proved the corresponding assertion in \cite[1.2]{carter}.
\end{proof}

\begin{corollary}\label{LF sequence}
If $F_1\to F_2 \to G\to H_1\to H_2$
is an exact sequence of functors and the $F_i$, $H_i$ are contracted
then $G$ is acyclic and there is an exact sequence for every $f$:
\[
LF_1(f) \to LF_2(f) \to LG(f) \to LH_1(f)\to LH_2(f).
\]
\end{corollary}

Of course, Corollary \ref{LF sequence} may be iterated to get
exact sequences for $LLG(f)$, etc., because the $LF_i$ and $LH_i$
are contracted functors.

\begin{example}\label{ex:Unil}
Recall from \cite[XII.7.8]{bass} \cite[III.4.1.3]{WK} that the
units $U$ form a contracted functor on rings with $LU(A)=H^0(\Spec A,\Z)$
and $LLU=NLU=0$. Similarly, $F(A)=U(A_\red)$ is a contracted functor
and its contraction $LF(A)=LU(A_\red)$ is isomorphic to $LU(A)$.
Define $\Unil(A)$ to be the kernel of $U(A)\to U(A_\red)$; 
it is the multiplicative group $(1+\textrm{nil}(A))^\times$. 
Since $1\to \Unil(A)\to U(A)\to U(A_\red)\to 1$ is an exact sequence,
Lemma \ref{mor} implies that $\Unil$ is a contracted functor with
$L\Unil(A)=0$.

Given a commutative ring extension $f:A\hookrightarrow B$,
define $\Unil(f)$ to be the cokernel of $\Unil(A)\to\Unil(B)$.
From the exact sequence $1\to\Unil(A)\to\Unil(B)\to\Unil(f)\to1$
and Lemma \ref{mor}, we see that $\Unil(f)$ is a contracted
functor with $L\Unil(f)=0$.
\end{example}

\begin{subremark}\label{rem:NI}
Since $N\Unil(f)=(1+t\,\nil(B)[t])^\times/(1+t\,\nil(A)[t])^\times$,
it follows that $N\Unil(f)=0$ if and only if $\nil(A)=\nil(B)$.
This is trivial if $B$ is reduced.
\end{subremark}

\section{Relative Cartier divisors}\label{sec:I}

Relative Cartier divisors are functors on ring extensions.
Recall that the $A$-submodules of $B$ form a monoid under
multiplication, with identity $A$. An $A$-submodule $L_1$ of $B$
is said to be {\it invertible} if $L_1L_2=A$ for some $L_2$.
In particular, $L_1$ is isomorphic to an invertible ideal of $A$.
An invertible  $A$-submodule is also said to be a
{\it relative Cartier divisor}.

\begin{definition}\label{def:cI}
Given a ring extension $f:A\hookrightarrow B$,
let $\cI(f)$ denote the multiplicative group of 
all $A$-submodules of $B$ which are invertible.  
We shall sometimes write $\cI(A,B)$ for $\cI(f)$.  It is easily 
seen that $\cI$ is a functor from the category of ring extensions 
to abelian groups. 
\end{definition}

The study of $\cI(A,B)$ was initiated by Roberts and Singh in 
\cite{rs}. 

If $\cO^\times_A$ is the Zariski sheaf of units on $\Spec(A)$
and $f_*\cO^\times_B$ is the direct image sheaf on $\Spec(A)$
associated to the units on $\Spec(B)$, it is easy to see that
\begin{equation}\label{I=H^0}
\cI(A,B) \cong H_\zar^0(\Spec A,f_*\cO^\times_B/\cO^\times_A).
\end{equation}
In effect, an invertible $A$-submodule $L$ can be described by
giving an open cover $\{U_i\}$, $U_i=\Spec(A[1/s_i])$ of $\Spec(A)$ and 
elements $f_i$ of $B[1/s_i]^\times$ (defined modulo $A[1/s_i]^\times$)
such that each $f_i/f_j$ is in $A[1/s_is_j]^\times$.

For example, if $A$ is a domain and $K$ is the field of fractions,
then $\cI(A,K)$ is the group of Cartier divisors and the
interpretation of $\cI(A,K)$ as $H^0(\Spec A,f_*\cO^\times_K/\cO^\times_A)$
is standard. For this reason, we shall call $\cI(f)$ the group of
{\it relative Cartier divisors.}

Since $\Pic(A)$ (the Picard group of $A$) is $H^1(\Spec A,\cO^\times_A)$,
and $H^1(\Spec A,f_*\cO^\times_B)$ is a subgroup of $\Pic(B)$,
the (Zariski or \'etale) cohomology sequence associated to the exact
sequence of sheaves on $\Spec A$,
$
1 \to \cO^\times_A \to f_*\cO^\times_B \to f_*\cO^\times_B/\cO^\times_A\to1
$
is the exact sequence mentioned in the Introduction:
\begin{equation}\label{seq:U-Pic}
 1\to  U(A) \to  U(B) \to  \cI(A,B) \to  \Pic A \to  \Pic B. 
\end{equation}
It is clear that this sequence is natural in $f$.
(A more elementary proof of exactness is given in \cite[Theorem 2.4]{rs}.)

\begin{proposition}\label{LU-LPic}
The functor $\cI$ is acyclic, $NL\cI=LL\cI=0$
and there is an exact sequence
\[
1 \to LU(A) \to LU(B) \to L\cI(f) \to L\Pic(A) \to L\Pic(B).
\]
\end{proposition}

\begin{proof}
The units $U$ and Picard group $\Pic$ are contracted functors on rings
(see \cite[5.2]{wei}).
Applying Corollary \ref{LF sequence} to \eqref{seq:U-Pic}, 
we see that $\cI$ is acyclic, and that there are exact sequences
\begin{align*}
1 \to LU(A) \to LU(B) \to& L\cI(f) \to L\Pic(A) \to L\Pic(B), \\
1 \to LLU(A) \to LLU(B) \to& LL\cI(f) \to LL\Pic(A) \to LL\Pic(B).
\end{align*}
Now $NLU=LLU=0$ by Example \ref{ex:Unil}, and
$LL\Pic=NL\Pic=0$ by \cite[7.7]{wei}.
This yields $LL\cI(f)=0$, $LU(A)=LU(A[t])$ and $L\Pic(A[t])=L\Pic(A)$.
It is immediate that $NL\cI(f)=0$.
\end{proof}

Since $L\Pic$ vanishes on normal domains \cite[1.5.2]{wei}, we
see that (i) if $A$ is a normal domain then $L\cI(f)=0$ if and only if 
$B$ is connected, and (ii) If $B$ is a normal domain then
$L\cI(f)=0$ if and only if $L\Pic(A)=0$.
More generally, we have:

\begin{corollary}\label{LPic.injects}
If $A$ is connected, and $f:A\hookrightarrow B$ is an extension, then
$L\cI(f)=0$ if and only if 
(i) $B$ is connected and
(ii) $L\Pic(A)\to L\Pic(B)$ is an injection.
\end{corollary}
\goodbreak

\begin{corollary}\label{tfree}
The group $L\cI(A, B)$ is always a torsion-free abelian group.
\end{corollary}

\begin{proof}
By \cite[2.3.1]{wei}, $L\Pic(A)$ is a torsion-free abelian group. 
In addition, the image of $LU(B)$ in $L\cI(A, B)$ is
free abelian by \cite[Prop.\,1.3]{wei}. The fact that
$L\cI(A, B)$ is torsionfree now follows from Proposition \ref{LU-LPic}.
\end{proof}


Recall from \cite{wei} that a noetherian ring $A$ is called
\textbf{pseudo-geometric} if every reduced finite $A$-algebra $B$ has
finite normalization. For example, any finitely generated algebra over
a field or over $\Z$ is pseudo-geometric.

\begin{proposition}\label{free}
If $A$ is pseudo-geometric and $\dim A\!<\!\infty$, then
$L\cI(A,B)$ is a free abelian group.
\end{proposition}

\begin{proof}
When $A$ is pseudo-geometric with $\dim A<\infty$,
$L\Pic A$ is a free abelian group by Proposition 2.3 of \cite{wei}.
So the image of $L\cI(f)$ in $L\Pic(A)$ is a free abelian group.
Again, Proposition \ref{LU-LPic} implies that $L\cI(A, B)$
is a free abelian group.
\end{proof}

\begin{subremark}
If $A$ is a 1-dimensional domain, then $L\cI(A, B)$ is a free
abelian group. This follows from the sequence of
Proposition \ref{LU-LPic}, and the facts that 
(i) $L\Pic(A)$ is a free abelian group \cite[3.4.1]{wei},
(ii) subgroups of free abelian groups are free and
(iii) the image of $LU(B)$ in $L\cI(A,B)$ is free abelian 
\cite[Prop.\,1.3]{wei}.
\end{subremark}

\begin{question}
Is $L\cI(A, B)$ always a free abelian group?
\end{question}

\medskip
For the rest of this paper, it is convenient to adopt 
scheme-theoretic language. 
Recall that a morphism of schemes $f:X\to S$ is {\it affine}
if the inverse image $f^{-1}U$ of any affine open subset $U$ of $S$ 
is an affine  open subset of $X$. We will say that an affine morphism is
{\it faithful} if $\cO_S\to f_*\cO_X$ is an injection; if the inverse
image of $\Spec(A)$ is $f^{-1}U=\Spec(B)$, this implies that $A\to B$
is an injection.

\begin{notation}
The category of ring extensions embeds contravariantly into the 
category of faithful affine morphisms of schemes, $f:X\to S$;
morphisms $f\to f'$ in this category are compatible pairs of maps
$X\to X'$ and $S\to S'$. If $f$ is a faithful affine morphism,
$\cI(f)$ will denote the multiplicative group of all
$\cO_S$-submodules of $f_*\cO_X$ which are invertible.
\end{notation}

It is clear that the formal yoga of Sections
\ref{sec:prelim} and \ref{sec:I} extend to the category of
faithful affine morphisms $X\to S$.
Given a faithful affine map $f:X\to S$, \eqref{I=H^0}
easily generalizes to
$\cI(f)\cong H_\zar^0(S,f_*\cO^\times_X/\cO^\times_S)$.
Proposition \ref{LU-LPic} implies that $\cI$ is an 
acyclic functor with $NL\cI=LL\cI=0$, 
Corollary \ref{tfree} states that $L\cI(f)$ is torsionfree.
Remark \ref{rem:NI} is replaced by: $N\Unil(f)=0$ if and only if 
$H^0(S,\nil\,\cO_S)=H^0(X,\nil\,\cO_X)$.

\bigskip\goodbreak
\section{Basic properties}\label{sec:basic}

In this short section, we give a few results that allow us to
relate $L\cI(f)$ to $L\cI$ of other ring extensions.
Given a map $f:A\hookrightarrow B$, we write
$f_\red$ for the evident map $A_\red\hookrightarrow B_\red$.

\begin{theorem}\label{LI-red}
The natural map $L\cI(A,B) \map{\cong} L\cI(A_\red,B_\red)$ is an
isomorphism. In addition, there is a 
natural short exact sequence of functors on ring extensions
\[
1\to\Unil(f)\to\cI(f)\to\cI(f_\red)\to 0.
\]
\end{theorem}

\begin{proof}
Consider the commutative diagram
\[\xymatrix @R=.25in {
1\ar[r]&\Unil(A) \ar[r]\ar@{>->}[d]&\Unil(B) \ar[r]\ar@{>->}[d]& 
\Unil(f) \ar[r]\ar[d]&1\\
1\ar[r]&U(A) \ar[r]\ar@{->>}[d] & U(B) \ar[r]\ar@{->>}[d] &\cI(f)\ar[r]\ar[d] &
\Pic(A) \ar[r]\ar@{=}[d] & \Pic(B)\ar@{=}[d] \\
1\ar[r]&U(A_\red)\ar[r]&U(B_\red)\ar[r]&\cI(f_\red)\ar[r]&
\Pic(A_\red)\ar[r]&\Pic(B_\red).
}\]
The groups $\Unil(A)$ and $\Unil(f)$ were defined in
Example \ref{ex:Unil}, where we
observed that the two left columns and the top row are 
short exact sequences; the bottom two rows are 
the exact sequences \eqref{seq:U-Pic}. Since $\Unil(A)$ is the 
intersection of $\Unil(B)$ and $U(A)$ in $U(B)$, a diagram chase 
shows that the third column is exact.

Since $L\Unil(f)=0$ by Example \ref{ex:Unil}, the isomorphism
$L\cI(f)\cong L\cI(f_\red)$ follows from the second part of 
Lemma \ref{mor}, applied to the third column.
\end{proof}

\begin{subremark}
The first part of Theorem \ref{LI-red} extends to $X\to S$
by our Main Theorem \ref{thm:main} below; see \ref{LI(f)red}.
The second part of Theorem \ref{LI-red} can fail for $X\to S$ as 
$U(S)\to U(S_\red)$ need not be onto.
\end{subremark}

\begin{corollary}\label{NI-NIred}
$N\cI(A,B)\map{\cong} N\cI(A_\red,B_\red)$ if and only if
$\nil(A)=\nil(B)$.

Moreover, if $\nil(A)\ne\nil(B)$ then $N\cI(f)\ne0$.
\end{corollary}

\begin{proof}
Replacing $f$ by $f[t]$ in Theorem \ref{LI-red}, we 
have the exact sequence
\[
1\to N\Unil(f)\to N\cI(f)\to N\cI(f_\red)\to 0.
\]
By Remark \ref{rem:NI}, the first term vanishes if and only if
$\nil(A)=\nil(B)$.
\end{proof}

\begin{lemma}\label{3rings-I}
(\cite[\S3]{bs})
Suppose that $f:A\hookrightarrow B$ and $g:B\hookrightarrow C$
are extensions. Then there is a short exact sequence
\[  1\to \cI(A,B) \to \cI(A,C) \to \cI(B,C). \]
\end{lemma}


\begin{proof}
We have an exact sequence of sheaves on $\Spec(A)$:
\[
1\to f_*\cO_B^\times/\cO_A^\times \to (fg)_*\cO_C^\times/\cO_A^\times
\to f_*(g_*\cO_C^\times/\cO_B^\times) 
\]
Now apply the left exact global sections functor and use \eqref{I=H^0}.
\end{proof}

\begin{lemma}\label{MV} 
Let $\ca$ be an ideal of $B$ contained in $A$. Then 
$L\cI(A, B)\cong L\cI(A/\ca, B/\ca)$.
\end{lemma}

\begin{proof}
Write $\bar{f}$ for $A/\ca\hookrightarrow B/\ca$.
By Proposition 2.6 of \cite{rs}, $\cI(A,B)\cong\cI(\bar{f})$.
Since $\ca[t]$ is an ideal of $B[t]$ contained in $A[t]$,
and $\ca[t,1/t]$ is an ideal of $B[t,1/t]$ in $A[t,1/t]$,
the same is true for $\cI(A[t], B[t])$ and $\cI(A[t,1/t], B[t,1/t])$.
The result follows from a comparison of $Seq(\cI,f)$ and $Seq(\cI,\bar{f})$.
\end{proof}

\smallskip
Here is another elementary result about $\cI$, which allows us
to assume for example that $A$ is noetherian and
$B$ is of finite type over $A$.

\begin{lemma}\label{colim}
$\cI$ commutes with filtered colimits. That is,
if $A\subset B$ is the filtered union of extensions 
$A_\lambda\subset B_\lambda$ then 
$\cI(A,B)=\varinjlim   \cI(A_\lambda,B_\lambda)$ and
$L\cI(A,B)=\varinjlim L\cI(A_\lambda,B_\lambda)$.
\end{lemma}

\begin{proof}
Since $U(B)=\cup U(B_\lambda)$ and $\Pic(B)=\varinjlim\Pic(B_\lambda)$,
this lemma follows from \eqref{seq:U-Pic}, $Seq(\cI,f)$
 and the fact that filtered direct limits are exact.
\end{proof}

\begin{proposition}
Suppose that $A=\prod_1^n A_i$ and $B=\prod_1^n B_i$,
where $A_i\subset B_i$.  
Then 
$\cI(A,B)=\prod\cI(A_i,B_i)$,
$N\cI(A,B)=\prod N\cI(A_i,B_i)$ and
$L\cI(A,B)=\prod L\cI(A_i,B_i)$.
\end{proposition}

\begin{proof}
Every $A$-submodule of $B$ has the form $M=\prod M_i$, where each
$M_i$ is an $A_i$-submodule of $B_i$. If $M$ is invertible with inverse
$N=\prod N_i$, then there are $m_j=(m_{ij})\in M$, $n_j=(n_{ij})\in N$
so that $\sum_j m_{ij}n_{ij}=1$ for all $i$. This shows that each
$M_i$ is an invertible $A_i$-submodule of $B_i$, and hence that
the natural map from $\cI(A,B)$ to $\prod\cI(A_i,B_i)$ is an 
injection. To see that it is a surjection, suppose that $M_i$ are 
invertible $A_i$-submodules of $B_i$ with inverses $N_i$. 
Then for each $i$ there are $m_{ij}$ and $n_{ij}$ so that
$\sum_j m_{ij}n_{ij}=1$. Thus
$\prod M_i$ is an invertible $A$-submodule of $B$.

Since $(\prod_1^n A_i)[t]=\prod_1^n(A_i[t])$,
the assertions about $N\cI$ and $L\cI$ follow by replacing
$A_i$ with $A_i[t]$ and $A_i[t,1/t]$, and similarly for $B_i$.
\end{proof}


\bigskip\goodbreak
\section{Main Theorem}\label{sec:main}

The goal of this section is to show that $\cI$ is a contracted functor,
whose contraction is an \'etale cohomology group.
We refer the reader to \cite{Milne} for basic properties of
\'etale sheaves and \'etale cohomology.

\begin{theorem}\label{thm:main}
$\cI$ is a contracted functor on ring extensions, and its contraction is
\[
L\cI(A,B)= H_\et^0(\Spec A,(f_*\Z)/\Z)=H_\nis^0(\Spec A,(f_*\Z)/\Z).
\]
\end{theorem}

Here $(f_*\Z)/\Z$ denotes the quotient sheaf in the \'etale topology.
Theorem \ref{thm:main} is just the special case 
$S=\Spec(A)$ and $X=\Spec(B)$ of the following result.

\begin{theorem}\label{thmSX}
$\cI$ is a contracted functor on faithful affine maps, 
with contraction 
\[
L\cI(f)= H_\et^0(S,(f_*\Z)/\Z) = H_\nis^0(S,(f_*\Z)/\Z).
\]
\end{theorem}


\begin{corollary}\label{LI(f)red}
$L\cI(f)\cong L\cI(f_\red)$.
\end{corollary}

We begin the proof of Theorem \ref{thmSX} by generalizing 
\eqref{I=H^0} to the \'etale and Nisnevich topologies on $S$.  
Recall that if $\cF$ is an \'etale sheaf on $S$ (a sheaf on $S_\et$)
then it is also a Nisnevich and a Zariski sheaf, and
$H_\et^0(S,\cF)=H_\nis^0(S,\cF)=H_\zar^0(S,\cF)=\cF(S)$.
This remark applies for example to the sheaves of units.
To avoid confusion,
it will be convenient to write $\cO^\times_S$ and $f_*\cO^\times_X$ 
for the \'etale sheaves $U\mapsto\Gamma(U,\cO_U)^\times$ and
$U\mapsto\Gamma(f^{-1}U,\cO_{f^{-1}U})^\times$,
instead of the traditional $\mathbb{G}_m$ and $f_*(\mathbb{G}_m|_{X_\et})$.
Of course, they are also sheaves for the Nisnevich topology on $S$.

\begin{lemma}\label{I=etaleH^0}
The Zariski quotient sheaf $f_*\cO^\times_X/\cO^\times_S$
is an \'etale sheaf. Consequently,
$\cI(f)\cong H_\et^0(S,f_*\cO^\times_X/\cO^\times_S)
       \cong H_\nis^0(S,f_*\cO^\times_X/\cO^\times_S)$.
\end{lemma}

\begin{proof}
Since $H_\zar^1(S,\cO^\times_S)=H_\et^1(S,\cO^\times_S)$
and $H_\et^1(S,f_*\cO^\times_X)$ is a subgroup of 
$H_\et^1(X,\cO^\times_X)$, we have a commutative diagram:
\[\xymatrix @C=.15in{
H_\zar^0(S,\cO^\times_S)    \ar[r]\ar@{=}[d] & 
H_\zar^0(S,f_*\cO^\times_X) \ar[r]\ar@{=}[d] & 
H_\zar^0(S,f_*\cO^\times_X/\cO^\times_S) \ar[r]\ar[d] &
H_\zar^1(S,\cO^\times_S)    \ar[r]\ar@{=}[d] & 
H_\zar^1(S,f_*\cO^\times_X)       \ar[d]_{\textrm{into}} \\
H_\et^0(S,\cO^\times_S)    \ar[r] & 
H_\et^0(S,f_*\cO^\times_X) \ar[r] & 
H_\et^0(S,(f_*\cO^\times_X/\cO^\times_S)_\et)\ar[r] &
H_\et^1(S,\cO^\times_S)    \ar[r] & H_\et^1(S,f_*\cO^\times_X).
}\]
From the 5-lemma, we see that the middle vertical map is an
isomorphism, i.e., that $f_*\cO^\times_X/\cO^\times_S$
is an \'etale sheaf, and hence a Nisnevich sheaf.
The final assertion follows from \eqref{I=H^0}.
\end{proof}

\begin{notation}\label{notation}
Given a scheme $S$, we write $S[t]$ for $S\times\Spec(\Z[t])$;
there is a natural map $p^{S,t}:S[t]\to S$. When the base $S$ is clear
we simply write $p^t$, so that $p^t_*\cO^\times_{S[t]}$
denotes the direct image sheaf on $S$; it is both a Zariski and an \'etale
sheaf on $S$. Similarly, we write $S[t,1/t]$ for 
$S\times\Spec(\Z[t,1/t])$, with projection $p:S[t,1/t]\to S$,
and also write $p_*\cO^\times_{S[t,1/t]}$ for the direct image sheaf on $S$.
Given $f:X\to S$ then, by abuse of notation, we will also write 
$f_*p^t_*\cO^\times_{X[t]}$ for the direct image sheaf on $S$ 
associated to the composition $X[t]\to X\to S$, etc.

For notational simplicity, we shall write $\cO^\times$ and 
$f^T_*\cO^\times$ for the \'etale sheaves $\cO^\times_{S[t,1/t]}$ and 
$f[t,1/t]_*\cO^\times_{X[t,1/t]}$ on $S[t,1/t]$. 
Thus Lemma \ref{I=etaleH^0} yields the formula 
\[
\cI(f[t,1/t])\cong H_\et^0(S[t,1/t],f^T_*\cO^\times/\cO^\times)
\cong H_\et^0(S,p_*(f^T_*\cO^\times/\cO^\times)).
\]
Replacing $H_\et^0$ with $H_\nis^0$ yields an analogous formula.
\end{notation}

\begin{subremark}\label{node}
The analogue of the formulas 
$\cI(f[t,1/t]) \cong H_\et^0(S,p_*(f^T_*\cO^\times/\cO^\times))$
and $L\cI(A,B)=H_\et^0(\Spec A,(f_*\Z/\Z))$
fail for the Zariski cohomology. To see this,
consider the subring $A$ of $B=k[x]$ defining the node.
It is not hard to see that 
\[
\cI(f[t,1/t])\cong\Pic(A[t,1/t]) = \Pic(A)\oplus\Z
\quad \textrm{and} \quad 
L\cI(f)\cong L\Pic(A)\cong\Z
\]
(see \cite[2.2]{wei}), yet 
$H_\zar^0(S,p_*(f^T_*\cO^\times/\cO^\times)_\zar)=\Pic(A)$
and $H_\zar^0(S,(f_*\Z/\Z)_\zar)=0$.

A similar calculation for the local ring $A_\wp$ of the node
and $B_\wp$ the (semilocal) normalization of $A_\wp$ shows that
$L\cI(A_\wp,B_\wp)=\Z$.
\end{subremark}

Recall that a local ring $A$ is {\it hensel} if every finite
$A$-algebra $B$ is a direct product of local rings.
A Nisnevich sheaf on $\Spec(A)$ is zero if and only if it is zero on 
$\Spec(A^h_\wp)$ for every prime ideal $\wp$, where $A^h_\wp$ is
the henselization of the local ring $A_\wp$.

\begin{lemma}\label{LI-stalk}
If $A$ is a hensel local ring then $L\cI(A,B)=H^0(\Spec B,\Z)/\Z$.
\end{lemma}

\begin{proof}
By \cite[2.5]{wei}, $L\Pic(A)=0$. Since $LU(A)=\Z$ and $LU(B)=H^0(\Spec B,\Z)$,
the sequence of Proposition \ref{LU-LPic} yields the result.
\end{proof}

\begin{subremark}\label{rem:hensel}
As noted in \cite[1.2.1]{wei},
$H^0(\Spec B,\Z)/\Z$ is a free abelian group for every $B$.
We saw in Corollary \ref{tfree} that $L\cI(f)$ is always torsionfree.
\end{subremark}

If we fix $f:X\to S$ and view $L\cI$ as the presheaf 
$U\mapsto L\cI(U,f^{-1}U)$ on the \'etale site of $S$, 
Lemma \ref{LI-stalk} says that
the associated \'etale sheaf is $f_*\Z/\Z$. Therefore we have
a canonical map $a_f:L\cI(f)\to H_\et^0(S,f_*\Z/\Z)$.

\begin{theorem}\label{LI=H^0}
The canonical map $a_f:L\cI(f)\to H_\et^0(S,f_*\Z/\Z)$
is an isomorphism.
\end{theorem}

\begin{proof}
We claim there is a commutative diagram whose rows are
the exact sequence of Proposition \ref{LU-LPic} and the cohomology sequence
associated to $\Z\to f_*\Z\to f_*\Z/\Z$:
\[\xymatrix{
LU(S) \ar[r]\ar@{=}[d] & LU(X) \ar[r]\ar@{=}[d] &
L\cI(f) \ar[r]\ar[d]^{a_f} & 
L\Pic(S) \ar[r]\ar@{=}[d] & L\Pic(X) \ar@{=}[d]
\\
H_\et^0(S,\Z) \ar[r] & H_\et^0(X,\Z) \ar[r]&
H_\et^0(S,f_*\Z/\Z) \ar[r] & H_\et^1(S,\Z) \ar[r] & H_\et^1(X,\Z).
}\]
Given this claim, the theorem follows from the 5-lemma.

The left three vertical maps are the canonical maps from the
evident presheaves to the global sections of the associated sheaves,
so the left two squares commute. Since the right two vertical maps
are the natural isomorphisms of \cite[5.5]{wei}, the right square
also commutes.  Thus we only need to show that the remaining square commutes.

Recall from \ref{notation} that $\cO^\times$ and $f^T_*\cO^\times$ 
are the \'etale sheaves $\cO^\times_{S[t,1/t]}$ and 
$f[t,1/t]_*\cO^\times_{X[t,1/t]}$ on $S[t,1/t]$.
The sheafification of $A[t,1/t]^\times\to H^0(S,\Z)$ on $S$ is a
map $\partial_S:p_*\cO^\times\to\Z$; it induces a map
$Rp_*\cO^\times\to\Z$ in the derived category of \'etale sheaves.
Similarly, the sheafification of $B[t,1/t]^\times\to H^0(S,f_*\Z)$ 
on $S$ induces a  map $f_*\partial_X:Rp_*(f^T_*\cO^\times)\to f_*\Z$.
Thus we have a morphism of triangles in the derived category.
\addtocounter{equation}{-1}
\begin{subequations}
\renewcommand{\theequation}{\theparentequation.\arabic{equation}}
\begin{equation}\label{eq:Rp_*}
\xymatrix{
Rp_*(\cO^\times) \ar[r]\ar[d]^{\partial_S}&
Rp_*(f^T_*\cO^\times) \ar[r]\ar[d]^{f_*\partial_X} &
Rp_*(f^T_*\cO^\times/\cO^\times) \ar[r]\ar[d] &
Rp_*(\cO^\times)[1] \ar[d]^{\partial_S} \\
\Z \ar[r] & f_*\Z \ar[r] & (f_*\Z)/\Z \ar[r] & \Z[1].
}\end{equation}
Note that $H_\et^0(S,Rp_*(\cO^\times)[1])=H_\et^1(S[t,1/t],\cO^\times)
=\Pic(S[t,1/t])$
and, by Lemma \ref{I=etaleH^0},  
\begin{equation}\label{eq:I(f[T])}
H_\et^0(S,Rp_*(f^T_*\cO^\times/\cO^\times)) =
H_\et^0(S[t,1/t],f^T_*\cO^\times/\cO^\times)=\cI(f[t,1/t]).
\end{equation}
\end{subequations}
Thus applying $H_\et^0$ to the right-hand square in \eqref{eq:Rp_*}
yields the commutative square
\[\xymatrix @R=.25in @C=.25in{
\cI(f[t,1/t]) \ar[r]\ar[d] & \Pic(S[t,1/t])\ar[d] \\
H_\et^0(S,f_*\Z/\Z) \ar[r] & H_\et^1(S,\Z).
}\]
The left map factors as 
$\cI(f[t,1/t])\to L\cI(f)\map{a_f} H_\et^0(S,f_*\Z/\Z)$,
and the right map factors as 
$\Pic(S[t,1/t])\to L\Pic(S)\cong H_\et^1(S,\Z)$.
The top map is the map in \eqref{seq:U-Pic}, fitting into the
commutative square
with surjective vertical maps
\[\xymatrix @R=.25in @C=.25in{
\cI(f[t,1/t]) \ar[r]\ar[d] & \Pic(S[t,1/t])\ar[d] \\
L\cI(f)\ar[r] &  L\Pic(S),
}\]
which is implicit in Proposition \ref{LU-LPic}. The claim follows.
\end{proof}

\begin{corollary}
The Nisnevich quotient sheaf $f_*\Z/\Z$ is an \'etale sheaf.
\end{corollary}

\begin{proof}
By Lemma \ref{LI-stalk}, it suffices to observe that if 
$S$ is hensel local we have $L\cI(f)=H^0_\et(S,f_*\Z/\Z)$.
\end{proof}

It remains to show that $\cI$ is a contracted functor.
Sheafifying the sequence $Seq(U,1_S)$ yields the sequence of sheaves on $S$:
\begin{equation}\label{seq:p_*}
1\to \cO^\times_S \map{\Delta} 
p^t_*(\cO^\times_{S[t]})\times p^{1/t}_*(\cO^\times_{S[1/t]}) 
\map{\pm} p_*(\cO^\times_{S[t,1/t]}) \map{\partial_S} \Z\to 1.
\end{equation}
In addition,
Bass' contraction $t_A:H^0(\Spec A,\Z)\to A[t,1/t]^\times$
is natural in $A$, so we may sheafify it to obtain a morphism 
$t_S:\Z\to p_*(\cO^\times_{S[t,1/t]})$ of (Zariski or \'etale) sheaves on $S$. 

\begin{lemma}\label{sheaf-Seq}
The sequence \eqref{seq:p_*} of sheaves on $S$ is split exact,
with splitting $t_S$.
\end{lemma}


\begin{proof}
On an affine open $\Spec(R)$ of $S$, this is just the sequence
\[
1 \to R^\times \map{\Delta} R[t]^\times\times R[1/t]^\times  
\map{\pm} R[t,1/t]^\times \map{\partial_R}\Z\to1.
\]
The fact that it is exact, and naturally split by $t_S$
is proven in \cite[XII.7.8]{bass}; see \cite[7.2]{wei}.
\end{proof}


\begin{corollary}\label{cor:f_*}
Given a faithful affine map $f:X\to S$,
the direct image of the sequence \eqref{seq:p_*} on $X$ 
is a split exact sequence of (Zariski or \'etale) sheaves 
on $S$, with splitting $f_*t_X$:
\[
1\to f_*\cO^\times_X \map{\Delta} 
f_*p^t_*(\cO^\times_{X[t]})\times f_*p^{1/t}_*(\cO^\times_{X[1/t]}) 
\map{\pm} f_*p_*(\cO^\times_{X[t,1/t]}) \map{f_*\partial_X} f_*\Z\to 1.
\]
\end{corollary}

The global sections of the sequences in \ref{sheaf-Seq} and \ref{cor:f_*}
are of course Bass' sequences $Seq(U,S)$ and $Seq(U,X)$.
By \ref{notation}, the sheafification of $\cI(f[t,1/t])$ on $S$
is $p_*(f^T_*\cO^\times/\cO^\times)$. Theorem \ref{LI=H^0} says that
the sheafification of $\cI(f[t,1/t])\to L\cI(f)$ on $S$
is a canonical map $\partial_f: p_*(f^T_*\cO^\times/\cO^\times) \to f_*\Z/\Z$.

\begin{proposition}\label{sheaf-split}
The map $\partial_f:p_*(f^T_*\cO^\times/\cO^\times) \to f_*\Z/\Z $ 
is split by a natural map of sheaves on $S$:
\[
t_f:f_*\Z/\Z\to p_*(f^T_*\cO^\times/\cO^\times).
\]
\end{proposition}

\begin{proof}
Applying the left exact functor $p_*$ to 
$1\to\cO^\times\to f^T_*\cO^\times \to f^T_*\cO^\times/\cO^\times\to1$,
we get exactness of the middle row in the following
commutative diagram of sheaves on $S$.
\[ \xymatrix@1{
0\ar[r]&\Z  \ar[r] \ar[d]^{t_S} & f_*\Z \ar[r] \ar[d]^{f_*t_X} & 
(f_*\Z)/\Z \ar[r] \ar@{.>}[d]^{t_f}& 0 \\
0\ar[r]&p_*\cO^\times \ar[r]\ar[d]^{\partial_S}& 
p_*f^T_*\cO^\times \ar[r]\ar[d]^{f_*\partial_X} 
& p_*(f^T_*\cO^\times/\cO^\times) \ar[r]\ar[d]^{\partial_f} 
& R^1p_*\cO^\times \\
0\ar[r]&\Z  \ar[r]  & f_*\Z \ar[r]  & (f_*\Z)/\Z \ar[r]& 0
}
\]
(The top and bottom rows are tautologically exact.)
The maps $t_S$, $f_*t_X$ induce the map $t_f$; since $\partial_S\, t_S$
and $\partial_X\, t_X$ are the identity, so are
$(f_*\partial_X)(f_*t_X)$ and $\partial_f\, t_f$.
\end{proof}

\begin{proof}[Proof of Theorem \ref{thmSX}]
By Theorem \ref{LI=H^0}, we have $L\cI(f)\cong H_\et^0(S,f_*\Z/\Z)$.
By Proposition \ref{sheaf-split}, we have a natural section $t_f$
of the sheaf map $\partial_f$. By \ref{notation}, 
the global sections of the map $\partial_f$ is the map
$\cI(f[t,1/t])\to L\cI(f)$ in $Seq(\cI,f)$. Hence the global sections
of $t_f$ provide the required natural splitting.
\end{proof}

Here is an easy consequence of Theorem \ref{thmSX},
which is related to Lemma \ref{3rings-I}.

\begin{corollary}\label{3rings}
Suppose that $f:A\hookrightarrow B$ and $g:B\hookrightarrow C$
are extensions. Then there is a short exact sequence
\[ 1 \to L\cI(A,B) \to L\cI(A,C) \to L\cI(B,C). \]
More generally, given faithful affine maps $X\map{g}T\map{f}S$,
there is an exact sequence
\[ 1 \to L\cI(f) \to L\cI(fg) \to L\cI(g). \]
\end{corollary}

\begin{proof}
Applying $f_*$ to the exact sequence
$0\to\Z\to g_*\Z\to (g_*\Z)/\Z\to 0$ on $T$ (or $\Spec(B)$) yields
the exact sequence of Nisnevich sheaves on $S$ (or $\Spec(A)$):
\[
1\to (f_*\Z)/\Z \to (fg)_*\Z/\Z \to f_*(g_*\Z/\Z).
\]
Now apply the left exact global sections functor and use 
Theorem \ref{thm:main}.
\end{proof}

\bigskip\goodbreak

\section{The Vanishing of $L\cI(A,B)$}\label{sec:vanish}

In this section, we discuss some conditions on $A\subset B$ under
which $L\cI(A,B)=0$. We begin by noting two elementary consequences
of the sheaf property of $f_*\Z/\Z$: (i) if $s,t\in A$ are
comaximal then $L\cI(A,B)\subset L\cI(A[\frac1s],B[\frac1s])\oplus 
L\cI(A[\frac1t],B[\frac1t])$, and 
(ii) if $L\cI(A_\wp,B_\wp)=0$ for every prime $\wp$ of $A$ then
$L\cI(A,B)=0$. The converse does not hold:

\begin{example}
If $A=\C[x]$ and $B=\C[x,y]/(y^2-x^2)$ then $L\cI(A,B)=0$,
but if $\wp\ne xA$ we have $L\cI(A_\wp,B_\wp)=\Z$
(use Proposition \ref{LU-LPic}).

If $A=k[s,s^{-1}]$ and $B=k[x,x^{-1}]$ with $s=x^2$ then
$L\cI(f)=0$ but $L\cI(A_\wp,B_\wp)=\Z$ for every nonzero
prime $\wp$ of $A$. Thus $(f_*\Z)/\Z$ is nonzero; its stalk is 
$\Z$ at any closed point, but is~0 at the generic point.
\end{example}

A simple necessary condition for $L\cI(f)$ to vanish is for the
Nisnevich sheaf $f_*\Z/\Z$ to vanish. It is not enough for
the Zariski sheaf $f_*\Z/\Z$ to vanish; Example \ref{node} shows 
that even if $A$ is a local ring we can have $(f_*\Z/\Z)_\zar=0$
and $H^0_\zar(A,f_*\Z/\Z)=0$ but
$L\cI(f)=H^0_\nis(A,f_*\Z/\Z)\ne0$. 

For finite morphisms, we have a simple criterion.
We say that a map $f:X\to S$ is {\it connected} if it for every point
$s$ of $S$, the fiber $X_s=f^{-1}(s)$ is connected or empty.
For a map $\Spec(B)\to\Spec(A)$, this means that for each prime ideal
$\wp$ of $A$ either there is no prime of $B$ over $\wp$ or else
the fiber ring $B\otimes_A k(\wp)$ is connected.


\begin{proposition}\label{connected}
Suppose that $f:X\to S$ is finite. Then

(a) the Nisnevich sheaf $(f_*\Z)/\Z$
is zero if and only if $f$ is connected.

(b) If $f$ is connected then $L\cI(f)=0$.  
\end{proposition}

\begin{proof}
Since the problem is local in $S$, we may suppose that $S=\Spec(A)$ and
$X=\Spec(B)$, with $A$ a local ring. Let $A^h$ be the henselization 
of $A_\wp$, and set $B'=B\otimes_A A^h$. 
Since $f$ is finite, $B'$ is a product of $n\ge1$ hensel local rings
$B_i$, each finite over $A^h$; see  \cite[1.4.2]{Milne}.
Since $B/\wp B = B'/\wp B' = \prod B_i/\wp B_i$, the fiber of $f$
at $\wp$ has $n$ components, and is connected iff $n=1$, i.e., iff
$B'$ is hensel local. Since the stalk of $f_*\Z/\Z$ at $\wp$
is zero iff $B'$ is hensel local, the result follows.
\end{proof}

\begin{examples}\label{X-node}
The hypothesis in \ref{connected} that $f$ be finite is necessary. 

(a) If $A=\Z$ and $B=\Z[\frac1p]\times\Z/p$ then $f:\Spec(B)\to\Spec(A)$
is quasi-finite and connected, but $L\cI(f)=\Z$ by 
Corollary \ref{LPic.injects}. Here $f_*\Z/\Z$ is a skyscraper
sheaf (at $p$).

(b) If $A$ is the coordinate ring $k[x,y]/(y^2-x^3-x^2)$ of the node,
and $B=A[1/x]$ then $A\subset B$ is \'etale and connected, yet
$L\cI(A,B)\cong L\Pic(A)\ne0$ by Corollary \ref{LPic.injects}. In this case,
$f_*\Z/\Z$ is the skyscraper sheaf $\Z$ at the nodal point.

(c) If $A=k[x]$ and $B=A[b,e]/(e^2-e-bx)$ then $f$ is not connected,
as $B/xB\cong k[b]\times k[b]$.
On the other hand, $f_*\Z/\Z=0$ and hence $L\cI(f)=0$. 
In fact, if $\wp\ne xA$ then $B\otimes A^h\cong A^h[e]$.
In this case, $f$ has relative dimension~1.
\end{examples}


\begin{examples}\label{not-conn}
Even if $f$ is finite but not connected we may still have $L\cI(f)=0$.

a) $L\cI(\R[x],\C[x])=0$, but $\R[x]\subset\C[x]$ is not connected.
In fact, $f_*\Z/\Z$ is nonzero exactly at those primes $\wp$ with
$\R[x]/\wp\cong\C$. This example shows that the rank of
$f_*\Z/\Z$ is not semicontinuous.

b)
If $A=k[x]$ and $B=k[x,y]/(y^2=x^3+x^2)$ is the coordinate ring of the node
then $L\cI(f)=H^0(\Spec(A),f_*\Z/\Z)=0$ but $(f_*\Z)/\Z$ is nonzero because
the stalk is $\Z$ at every point except at $x=0,-1$ and 
at the generic point (where the stalks are~0).
\end{examples}

We now turn to the connection between $L\cI$ and seminormalization.

\begin{definition}\label{def:sn} (Swan \cite[\S2]{swan}) 
An extension $A\subset B$ is {\it subintegral} if $B$ is integral over $A$, 
and $\Spec(B)\to\Spec(A)$ is a bijection inducing isomorphisms on all 
residue fields. 

We say that $A$ is {\it seminormal} in $B$ if whenever 
$b\in B$ and $b^{2}, b^{3}\in A$ then $b\in A$. 
The {\it seminormalization} of $A$ in $B$
is the largest subring $\scA_B$ of $B$ which is subintegral over $A$.
By \cite[2.5]{swan}, $\scA_B$ is seminormal in $B$.
\end{definition}

These notions extend to faithful affine maps of schemes in the evident way; 
the seminormalization of $S$ in $X$ may be constructed by gluing
together the seminormalizations on each affine open.
We omit the details.

\begin{subremark}\label{NI=0}
The condition that $N\cI(A,B)=0$ is equivalent to
$A$ being seminormal in $B$, and implies that $N^i\cI(A,B)=0$ 
for all $i>0$; this was proven in \cite[1.5, 1.7]{ss}.
More generally, a faithful affine map $f:X\to S$ is seminormal
if and only if $N\cI(f)=0$. This follows
from the affine case, since both $N\cI$ and 
seminormality are Zariski-local on $S$.
\end{subremark}

\begin{proposition}\label{LI-sn}
$L\cI(A,\scA_B)=0$ and 
$L\cI(A,B)\cong L\cI(\scA_B,B)$.  
\end{proposition}

\begin{proof}
The first assertion follows from Proposition \ref{LU-LPic},
since $LU(A)=LU(\scA_B)$ (because $\Spec(A)\to\Spec(\scA_B)$ is a bijection)
and $L\Pic(A)=L\Pic(\scA_B)$, by \cite[5.4]{wei}. (The hypothesis
in \cite[5.4]{wei} that $A$ be reduced is not needed in its proof.)
\edit{not needed!}

Now $\cI(A[t,1/t],B[t,1/t])\to\cI(\scA_B[t,1/t],B[t,1/t])$
is onto by \cite[4.1]{vs}. Hence the map $L\cI(A,B)\to L\cI(\scA_B,B)$
is onto. By Corollary \ref{3rings}, the kernel is $L\cI(A,\scA_B)=0$.
\end{proof}

\begin{definition} (Asanuma)
A ring extension $A\subset B$ is called {\it anodal} 
if every $b\in B$ such that $(b^{2}- b)\in A$ and $(b^{3}- b^{2})\in A$ 
belongs to $A$. 
\end{definition}

If $A\subset B$ is anodal then every idempotent of $B$ belongs to $A$,
so $H^0(A,\Z)=H^0(B,\Z)$. In particular, every anodal extension of
a domain is connected. If $A$ is a field then
$A\subset B$ is anodal if and only if $B$ is connected.
The first author proved that the composition of anodal extensions 
is anodal; see \cite[3.1]{vs}. 

The following result generalizes a result of Asanuma 
(see \cite[3.4]{wei}), who considered the case $B=\textrm{frac}(A)$,
as well as several results of the first author in \cite{vs}.

\begin{theorem}\label{ano}
Let $A\subset B$ be an extension. 
 \begin{enumerate}[(1)]
 \item If $L\cI(A, B)=0$ then $A\subset B$ is anodal.
 \item If $A$ is a 1-dimensional domain, and $A\subset B$ is
an integral, birational and anodal extension, 
then $L\cI(A,B)=0$.
\end{enumerate}
\end{theorem}

Example \ref{X-node}(b) 
shows that the integral hypothesis is necessary in
Theorem \ref{ano}(2).  Example \ref{2 dim} shows 
that not all integral, birational anodal extensions have $L\cI(f)\!=\!0$.

\begin{proof} (cf.\ Onoda-Yoshida \cite[1.10]{OY})
Let $b\in B$ be such that $b^2-b, b^3-b^2$
are in $A$; we need to show that $b\in A$. Consider the
finite subring $C=A[b]$ of $B$.  
If $L\cI(A,B)=0$, then $L\cI(A,C)=0$ by Corollary \ref{3rings}.
Let $\ca$ denote the ideal $(b^2-b)C$ of $C$; 
it is also an ideal of $A$, so $L\cI(A/\ca,C/\ca)=0$
by Lemma \ref{MV}. By Proposition \ref{LU-LPic}, 
this implies that $H^0(A/\ca,\Z)\cong H^0(C/\ca,\Z)$.
Since the image $\bar{b}$ of $b$ is idempotent in $C/\ca$,
this forces $\bar{b}\in A/\ca$ and hence $b\in A$,
proving (1).

(2) Now suppose that $A$ is a 1-dimensional domain, and that
$A\subset B$ is an integral, birational and anodal extension.
Since $B$ is the union of finite $A$-algebras
$B_\lambda$, all of which are anodal over $A$,
we may assume that $B$ is a finite $A$-algebra 
by Lemma \ref{colim}. Since $A\subset B$ is finite and birational, 
the conductor ideal $\cc$ is nonzero, so $\dim A/\cc=0$. 
By \cite[3.6]{wei}, the extension $\bar{f}:A/\cc \subset B/\cc$
is anodal because $A\subset B$ is. In particular,
$\bar{f}$ is connected. Since $\dim(A/\cc)=0$, 
$L\Pic(A/\cc)=0$ (by \cite[1.6.1]{wei}) and hence $L\cI(\bar{f})=0$ by
Proposition \ref{LU-LPic}. By Lemma \ref{MV},
$L\cI(A,B)=L\cI(A/\cc,B/\cc)=0$.
\end{proof}

\goodbreak

Here is an example of a 2-dimensional integral, birational and
anodal extension with $L\cI(A,B)\ne0$, 
Thus Theorem \ref{ano}(2) does not extend to $\dim(A)>1$. 

\begin{example}\label{2 dim}
Let $X$ be the coordinate axes in the plane, and $f:X\to S$ the
quotient identifying each axis with the normalization of the node $S$.
Consider the pushout $S\to S'=\Spec(A)$ of the tautological inclusion
of $X$ in $\mathbb{A}^2=\Spec(B)$ along $f$.

The map $\mathbb{A}^2\to S'$ is the map $\Spec(B)\to\Spec(A)$ of
Example 3.5 in \cite{wei}. By construction, $A$ is a 2-dimensional 
domain whose integral closure is $B=k[x,y]$, so 
$A\subset B$ is an integral, birational extension. 
It is shown in  \cite[3.5.2]{wei} that $A\subset B$ is anodal
and $L\Pic A=\Z$. Since $B$ is normal, $L\Pic B=0$.
Since $A$, $B$ are domains, $LU(A)=LU(B)=\Z$ 
(by Example \ref{ex:Unil}).
By Proposition \ref{LU-LPic}, $L\cI(A,B)\cong L\Pic(A)\cong\Z$.
\end{example}

\end{document}